\documentclass[11pt,a4paper]{amsart}
\usepackage[a4paper, margin=3.06cm]{geometry}

\title{Tracially $\mathcal{Z}$-absorbing $C^*$-algebras}
\author{Ilan Hirshberg}
\thanks{This research was supported in part by Israel Science Foundation grant 1471/07}
\email{ilan@math.bgu.ac.il}
\address{Department of Mathematics, Ben Gurion University of the Negev, P.O.B. 653, Be'er Sheva 84105, Israel}
\author{Joav Orovitz}
\email{joav@post.bgu.ac.il}
\address{Department of Mathematics, Ben Gurion University of the Negev, P.O.B. 653, Be'er Sheva 84105, Israel}

\usepackage{amssymb}
\usepackage{amsmath}
\usepackage{amsthm}
\usepackage[retainorgcmds]{IEEEtrantools}

\theoremstyle{plain}

\newtheorem{Thm}{Theorem}[section]
\newtheorem{Cor}[Thm]{Corollary}
\newtheorem{Lemma}[Thm]{Lemma}

\newtheorem{Prop}[Thm]{Proposition}

\theoremstyle{definition}
\newtheorem{Def}[Thm]{Definition}
\newtheorem{Notation}[Thm]{Notation}

\newtheorem{Exl}[Thm]{Example}
\newtheorem{Rmk}[Thm]{Remark}

\newcommand{\B}{\mathcal{B}}
\newcommand{\A}{\mathcal{A}}

\newcommand{\Zh}{\mathcal{Z}}

\newcommand{\Rr}{\mathcal{R}}

\newcommand{\N}{{\mathbb N}}
\newcommand{\Z}{{\mathbb Z}}
\newcommand{\C}{{\mathbb C}}

\newcommand{\lb}{\left <}
\newcommand{\rb}{\right >}

\newcommand{\aut}{\mathrm{Aut}}
\newcommand{\ad}{\mathrm{Ad}}

\newcommand{\lnorm}{\left \|}
\newcommand{\rnorm}{\right \|}

\newcommand{\eps}{\varepsilon}

\newcommand{\TZA}{tracially $\Zh$-absorbing}
\newcommand{\Csalg}{$C^*$-algebra}

\begin{document}
\begin{abstract}
We study a tracial notion of $\Zh$-absorption for simple, unital \Csalg s. We show that if $\A$ is a \Csalg~  for which this property holds then $\A$ has almost unperforated Cuntz semigroup, and if in addition $\A$ is nuclear and separable we show this property is equivalent to having $\A\cong\A\otimes\Zh$. We furthermore show that this property is preserved under forming certain crossed products by actions satisfying a tracial Rokhlin type property.
\end{abstract}
\maketitle

\section{Introduction}
The purpose of this paper is to introduce and study a property of $C^*$-algebras which can be thought of as a tracial version of $\Zh$-absorption, in a way reminiscent of the definition of tracially AF algebras and $C^*$-algebras of higher tracial rank (see \cite{lin}). There are several motivations for looking at this property. A property closely related to ours (the two coincide in certain cases) appeared as a technical step in Winter's work concerning $\Zh$-absorption for $C^*$-algebras of finite nuclear dimension (\cite{winter-dr-Z,winter3}) and the recent work of Matui-Sato on strict comparison and $\Zh$-absorption (\cite{matui-sato}), and it may thus be profitable to isolate this property for further study. Additionally, this property may be easier to establish in certain instances, in particular when considering crossed products by an action which satisfies a tracial version of the Rokhlin property. Indeed, we study a generalization of the tracial Rokhlin property which does not require projections and show that under certain conditions, crossed products of tracially $\Zh$-absorbing $C^*$-algebras by such actions are again tracially $\Zh$-absorbing. The Rokhlin property has been instrumental in the study of group actions on $C^*$-algebras (see \cite{Iz0} for a survey, \cite{Iz,Iz2,osaka-phillips} for the finite group case, \cite{Ks1} and references therein for the single automorphism case). However, the Rokhlin property might be harder to establish in certain cases, and might not exist in others. In the finite group case, the Rokhlin property is uncommon and its existence requires restrictive $K$-theoretic constraints on the algebra and the action. While for the single automorphism case the Rokhlin property is much more common (and even generic in certain cases, see e.g. \cite{hwz}), it still requires the existence of many projections. The Rokhlin property has been generalized to a tracial version in \cite{OP-tracial, phillips}, although this generalization still requires the existence of projections. Further tracial-type generalizations in which the projections are replaced by positive elements were considered in \cite{sato,Archey}, and we will study a closely related variant. We note that a different generalization which does not require projections has been to view the Rokhlin property as a zero-dimensional level of what can be called \emph{Rokhlin dimension}, see \cite{hwz}.

A unital \Csalg~ $\A$ is $\Zh$-absorbing if and only if for any $n\in\N$, any finite subset $F\subseteq\A$ and for any $\eps>0$ there exist c.p.c. order zero maps $\varphi:M_n\to \A$ and $\psi:M_2:\to \A$ such that the image of any normalized matrix under $\varphi$ or $\psi$ commutes up to $\eps$ with the elements of $F$, and such that $\psi(e_{1,1})\varphi(e_{1,1})= \psi(e_{1,1})$ and $\psi(e_{2,2}) = 1_\A - \varphi(1)$ (see \cite{rordam-winter}). In this characterization, $1-\varphi(1)$ must be ``small'' in a tracial sense. In this paper we introduce a property called \emph{tracial $\Zh$-absorption} which basically amounts to dropping the requirement for an order zero map from $M_2$ into $\A$ and instead asking that $1-\varphi(1)$ is arbitrarily small in the sense of Cuntz comparison (see Definition \ref{def:TZA}). 
We show that \TZA~ \Csalg s have almost unperforated Cuntz semigroups. This generalizes  results of R{\o}rdam in \cite{rordam} concerning $\Zh$-absorption. We then use ideas from \cite{matui-sato} to show that tracial $\Zh$-absorption implies $\Zh$ absorption for simple, unital, separable, nuclear \Csalg s. 

The paper is organized as follows. In section \ref{section:basic-properties} we define tracial $\Zh$-absorption, establish notation, and prove some basic properties. Section \ref{section:almost-unperforation} is devoted to showing that \TZA~ \Csalg s have almost unperforated Cuntz semigroup (and thus strict comparison). In section \ref{section:nuclear} we use this together with the techniques of Matui and Sato to show that a simple, unital, separable, nuclear \TZA~ $C^*$-algebra is in fact $\Zh$-absorbing. In section \ref{section:finite-group} we define a projectionless version of the tracial Rokhlin property for finite group actions and we show that tracial $\Zh$-absorption passes to crossed products under such actions. In section \ref{section:single-auto} we introduce a similar Rokhlin type property for automorphisms (i.e. integer actions), closely related to the one considered by Sato in \cite{sato}, and prove that tracial $\Zh$-absorption is preserved by crossed products by such automorphisms, assuming that some power of the automorphism acts trivially on the tracial state space. This technical assumption holds automatically in many cases (e.g. if some power of the automorphism is approximately inner or if the algebra has finitely many extreme traces). 

We thank Wilhelm Winter for some helpful conversations regarding this paper.

\section{Basic properties of tracially $\Zh$-absorbing algebras}
\label{section:basic-properties}

Recall that a c.p.c. map $\varphi:\A \to \B$ is said to have \emph{order zero} if $\varphi(a) \varphi(b)=0$ whenever $a, b \in \A_+$ satisfy $ab=0$. For positive elements $a,b \in \A$ we say that $a$ is \emph{Cuntz-subequivalent} to $b$ (written $a\precsim b$) if there is a sequence $x_n\in \A$ such that $\lim_{n\to\infty}\|a-x_nbx_n^*\| = 0$.
\begin{Def}
\label{def:TZA}
We say that a unital \Csalg~ $\A$ is \emph{\TZA} if $\A \ncong \C$ and for any finite set $F \subset \A$, $\eps>0$ and non-zero positive element $a \in \A$ and $n \in \N$ there is an order zero contraction $\psi:M_n \to \A$ such that the following hold:
\begin{enumerate}
\item $1-\varphi(1) \precsim a$. 
\item For any normalized element $x \in M_n$ and any $y \in F$ we have $\|[\varphi(x),y]\| < \eps$.
\end{enumerate}
\end{Def}

\begin{Prop}
Let $\A$ be a simple unital \Csalg. If $\A$ is $\Zh$-absorbing then $\A$ is \TZA.
\end{Prop}
\begin{proof}
First assume $\A$ is stably finite. By \cite{rordam} we know that $\A$ has strict comparison. Let $a, \eps, n, F$ be given as in Definition \ref{def:TZA}. Since $\A$ is simple there exist $\ell\in \N$ and $c_1,\ldots,c_\ell\in\A$ such that $$\sum_{k=1}^\ell c_k^*ac_k = 1.$$ This clearly implies that $d_\tau(a)\geq 1/\ell$ for all $\tau\in T(\A)$. Let $m$ be some number such that $m > \ell$ and $n$ divides $m$. Since $\A$ is $\Zh$-absorbing there is a unital homomorphism $\psi:\Zh_{m,m+1}\to \A$ such that for any normalized element $x\in\Zh_{m,m+1}$ and for any $y\in F$ we have $$\|[\psi(x),y]\|<\eps.$$  By \cite{rordam-winter} we have a c.p.c. order zero map $\tilde\varphi:M_m\to x\in\Zh_{m,m+1}$ such that $$1-\tilde\varphi(1) \precsim \tilde\varphi(e_{1,1}).$$ Let $\varphi = \psi\circ\tilde\varphi:M_m\to \A$. For any $\tau\in T(\A)$ we have that
$$d_\tau(1-\varphi(1)) \leq d_\tau(\varphi(e_{1,1})) \leq 1/m \leq d_\tau(a)$$
which entails that $1-\varphi(1)\precsim a$. Since $M_n$ embeds unitally into $M_m$ we restrict $\varphi$ to obtain the map we are looking for.

If $\A$ is not stably finite then, by \cite{rordam}, $\A$ is purely infinite. Set $m=n$ and continue as before. The condition $1-\varphi(1)\precsim a$ is automatically satisfied.
\end{proof}

In what follows we identify $M_n(\A)$ with $M_n \otimes \A$ in the usual way.

\begin{Lemma}
Let $\A$ be a simple, unital, non type $I$ \Csalg~ and let $n \in \N$. For every non-zero positive element $a\in M_n(\A)$ there exists a non-zero positive element $b\in \A$ such that $a \succsim 1\otimes b$.
\end{Lemma}

\begin{proof}
Since $a$ is positive and non-zero, there exists $i$ such that $(e_{ii}\otimes 1_\A) a (e_{ii}\otimes 1_\A)\neq 0$. Assume without loss of generality that $i=1$. Furthermore, by replacing $a$ with the element $(e_{1,1}\otimes 1_\A) a (e_{1,1}\otimes 1_\A)$ we may assume that $a$ is of the form $e_{1,1} \otimes c$ for some non-zero positive $c\in \A$.

Note that both $\A$ and $c\A c$ have a strictly positive element and hence, by Brown's Theorem, we have that $\overline{c\A c}$ is stably isomorphic to $\A$, this implies that the former is a simple, infinite dimensional, \Csalg~ not isomorphic to the compact operators. In particular $\overline{c\A c}$ is non-type $I$. Therefore, by Proposition 4.10 of \cite{kirchberg-rordam} (which is a corollary of Glimm's Theorem) there is a non-zero homomorphism $\theta:CM_n \to \overline{c\A c}$. Let $z\in C_0((0,1])$ denote the identity function. Using the picture $CM_n = C_0((0,1]) \otimes M_n$ we denote $h=\theta (z\otimes 1)$ and $b = \theta(z\otimes e_{1,1})$. Observe that $e_{1,1} \otimes h \sim 1\otimes b$ in $M_n(\A)$. Consequently we have that $a \succsim 1\otimes b$ as needed. 
\end{proof}

\begin{Lemma}\label{lemma:matrixtza}
Let $\A$ be a simple unital \Csalg. If $\A$ is \TZA~ then so is $M_n(\A)$ for any $n$.
\end{Lemma}
\begin{proof}

It is clear that $M_n(\A) \cong M_n \otimes \A$ satisfies the conditions of Definition \ref{def:TZA} provided the positive element $a$ is of the form $1 \otimes b$, $b \in \A$. Thus, it suffices to show that every non-zero positive element $a \in M_n(\A)$ Cuntz-dominates a non-zero positive element of the form $1 \otimes b$. But this follows from the previous lemma.
\end{proof}

\begin{Notation}\label{notation:centralsequencealgebra}
Let $\A$ be a separable \Csalg. We denote
$$\A_\infty = \prod_\N \A / \bigoplus_\N \A$$
We view $\A$ as embedded into $\A_\infty$ as equivalence classes of constant sequences and we denote by $$\A_\infty \cap \A^\prime$$ the relative commutant of $\A$ in $\A_\infty$.
\end{Notation}

The following result can help simplify proofs.
\begin{Lemma}\label{lemma:centralsequencealgebra}
Let $\A$ be a separable, unital \Csalg. If $\A$ is \TZA~ then for any $n\in \N$ and any non-zero positive contraction $a\in \A$ there exists a c.p.c. order zero map $\varphi:M_n\to \A_\infty \cap \A^\prime$ such that $1_{\A_\infty}-\varphi(1)\precsim a$ in $\A_\infty$.
\end{Lemma}
\begin{proof}
Let $(b_k)_{k\in\N}\subseteq \A$ be a dense sequence and denote $F_k=\{b_1,\ldots,b_k\}$. For each $k\in \N$ find a c.p.c. order zero map $\psi_k:M_n \to \A$ such that $1-\psi_k(1)\precsim a$ and $\|[\psi(x), y]\|<\frac{1}{k}$ for all normalized $x\in M_n$ and for all $y\in F_k$. Let $\varphi$ be the composition of the map $$(\psi_k)_{k\in \N}: M_n \to \prod_\N \A$$ with the quotient map into $\A_\infty$. It is easy to see that $\varphi$ has the desired properties.
\end{proof}

The following is a restatement of Corollary 4.2 of \cite{winter-zacharias2} along with some additional notation:

\begin{Notation}
Let $\varphi:F \to \A$ be an order zero c.p.c. map, where $F$ is finite dimensional. Denote $h=\varphi(1)$. Recall that there is a homomorphism $\pi:F \to \A^{**}\cap \{h\}^\prime$ such that $\varphi(x) = \pi(x)h$ where $h=\varphi(1)$.

If $f \in C_0((0,1])_+$, we denote by 
$$f[\varphi]$$ the c.p. order zero map given by 
$$f[\varphi](x) = \pi(x) f(h).$$ 
If $f$ is the square root function, we may use the notation $\sqrt{[\varphi]}$ for $f[\varphi]$.

For a c.p.c. map $\varphi:F \to \A$ from a finite dimensional algebra $F$, and an element $b \in \A$, we write  $$\|[\varphi,b]\| < \eps$$ to mean $\| [\varphi(x),b]\| < \eps$ for all normalized elements $x\in M_n$. If $X \subseteq \A$ is some subset, we shall use the notation 
 $$\|[\varphi,X]\| < \eps$$ to mean
$\|[\varphi,b]\| < \eps$ for all $b \in X$.

\end{Notation}

The next lemma is a special case of Proposition 1.9 of \cite{winter-dr-Z}

\begin{Lemma}
\label{lemma:orderzerofunctionalcalculus}
For any $f \in C_0((0,1])_+$ and any $\eps>0$ there is an $\eta>0$ such that whenever $\varphi:M_n \to \A$ is a c.p.c. order zero map and $b$ is a contraction satisfying $\|[\varphi,b]\|<\eta$, we have that $\|[f[\varphi],b]\|<\eps$.
\end{Lemma}

\begin{proof}
Let $f$ and $\eps$ as above be given. We can find $g\in C_0((0,1])$ such that $\|gz-f\|<\frac{\eps}{3}$ where $z$ is the identity function on $(0,1]$. Find $\eta>0$ such that $\|[a,b]\|<\eta$ implies $\|[g(a),b]\|<\frac{\eps}{6}$ for any normalized elements $a,b\in \A$ with $a$ positive (this is easily done by approximating $g$ uniformly by polynomials). We may of course assume that $\eta<\frac{\eps}{6}$. Let $\varphi:M_n\to \A$ be a c.p.c. order zero map and $b\in \A$ a contraction such that $\|[\varphi,b]\|<\eta$. Let $\pi,h$ be as in the notation above. For any normalized $x\in M_n$ we have the following:
\begin{align*}
f[\varphi](x)b &= \pi(x)f(h)b \\
& \approx_{\eps/3} \pi(x)hg(h)b \\
& = \varphi(x)g(h)b \\
& \approx_{\eps/3} b\varphi(x)g(h) \\
& \approx_{\eps/3} b f[\varphi](x).
\end{align*}
\end{proof}

We will also be needing a slight variation of this previous lemma. We omit the proof since it is essentially the same as that of the preceding lemma.

\begin{Lemma}
\label{lemma:orderzerofunctionalcalculus2}
For any $f \in C_0((0,1])_+$ and any $\eps>0$ there is an $\eta>0$ such that whenever $\varphi:M_n \to \A$ is a c.p.c. order zero map and $\alpha\in \aut(\A)$ is an automorphism satisfying $\|\alpha(\varphi(x))-\varphi(x)\|<\eta$ for all normalized elements $x\in M_n$, we have that $\|\alpha(f[\varphi](x))-f[\varphi](x)\|<\eps$ for all normalized $x\in M_n$.
\end{Lemma}

\section{Almost Unperforation of \TZA~ \Csalg s}
\label{section:almost-unperforation}

In this section we will prove that \TZA~ \Csalg s have almost unperforated Cuntz semigroup and therefore strict comparison. The proof mixes ideas from \cite{rordamuhf2} and \cite{rordam} with ideas originating from the study of tracially AF algebras (\cite{lin}).

Recall that a positive element $a \in \A$ is called \emph{purely positive} if $a$ is not Cuntz-equivalent to a projection. This is equivalent to saying that $0$ is an accumulation point of $\sigma(a)$.

\begin{Lemma}
\label{lemma:notprojection} 
Let $\A$ be a simple non type $I$ \Csalg. For any projection $p \in \A$ and $k>0$ there is a purely positive element $a \leq p$ such that $(k-1)\lb p \rb \leq k \lb a \rb$. 
\end{Lemma}
\begin{proof}
Since $\A$ is simple and non type $I$, the algebra $\B = \overline{p\A p}$ is also non type $I$ (as in the proof of Lemma \ref{lemma:matrixtza}). Hence, by Proposition 4.10 of \cite{kirchberg-rordam} there is an injective homomorphism $\theta:CM_k \to \B$. Let $z$ denote the identity function on $(0,1]$ and let $c_i = \theta(e_{ii}\otimes z)$. Let $a = p-c_1$. It is clear by functional calculus (and because $\theta$ is injective) that $a$ is purely positive.

Additionally, we have that $k\lb a \rb$ is represented by $1_k \otimes a \in  M_k \otimes \B$, and notice that 
$$1_k \otimes a \geq (1_k-e_{11}) \otimes a + e_{11} \otimes \sum_{j=2}^{k}c_i.$$ 
The Cuntz class of the right hand side is $(k-1)\lb a \rb + (k-1) \lb c_1 \rb$, which dominates $(k-1) \lb a+c_1 \rb = (k-1) \lb p \rb$.
\end{proof}

\begin{Lemma}
\label{lemma:almostunperf}
Let $\A$ be a unital \TZA~ \Csalg. If $a,b \in \A$ such that $k\lb a \rb \leq k \lb b \rb$ in $W(\A)$ for some $k \in \N$ and $b$ is purely positive then $a \precsim b$.
\end{Lemma}
\begin{proof}
Fix $\eps>0$. Without loss of generality we may assume that $\lnorm a \rnorm = \lnorm b \rnorm = 1$. We choose $c = (c_{ij}) \in M_k(\A)$ and $\delta>0$ such that $c \left[\left(b-\delta\right)_+ \otimes 1_k\right] c^* = \left(a-\eps\right)_+ \otimes 1_k$. Let $f\in C_0((0,1])$ be a non-negative function such that $f=0$ on $(\delta/2,1]$, $f>0$ on $(0,\delta/2)$ and $\lnorm f \rnorm = 1$ and denote $d = f(b)$. Note that $d \neq 0$ because $b$ is purely positive. We now fix $\mu>0$. We shall find $z \in \A$ such that $\lnorm z\left[\left(b-\delta\right)_++d\right]z^*-\left(a-\eps\right)_+\rnorm<\mu$. Since $\mu$ is arbitrary this will show that $\left(a-\eps\right)_+\precsim \left(b-\delta\right)_++d \precsim b.$
By replacing each $c_{ij}$ with $c_{ij}q(b)$ for some function $q\in C_0((0,1])$ that vanishes on $(0, \delta/2]$ and $q=1$ on $[\delta, 1]$ we may assume that $c_{ij}d = 0$. Note that after doing this we still have $c \left[\left(b-\delta\right)_+ \otimes 1\right] c^* = \left(a-\eps\right)_+ \otimes 1_k$ which yields

\begin{align}
\sum_{\ell=1}^{k}c_{i\ell}\left(b-\delta\right)_+c_{j\ell}^* =
\left \{
\begin{matrix}
	(a-\eps)_+ & \mid & i=j \\
	0 			& \mid & i \neq j 
\end{matrix}  \label{eq:1.1}
\right .
\end{align}

Let $g,h\in C_0((0,1])$ be defined by
\begin{align*}
g(t)&= \begin{cases}\sqrt{\frac{7}{\mu}t}, & t\leq \frac{\mu}{7} \\ 1, & t\geq\frac{\mu}{7} \end{cases}, & h(t) &= 1-\sqrt{1-t}
\end{align*}
We observe that those functions satisfy the following.
\begin{align} \label{eq:1.2}
|g(t)^2t-t|&<\frac{\mu}{6}, &1-h(t)=&\sqrt{1-t} 
\end{align}
in the algebra $C_0((0,1])$.

Find a c.p.c order zero map $\varphi: M_k \to \A$ such that $1-\varphi(1) \precsim d$ and $\lnorm [\varphi, F] \rnorm < \eta$ where $F = \left\{\left(a-\eps\right)_+\, , \left(b-\delta\right)_+\right\} \cup \{c_{ij}\}$, and where $\eta>0$ is chosen using Lemma \ref{lemma:orderzerofunctionalcalculus} such that
\begin{align*}
\|[g[\varphi],F]\|&<\frac{\mu}{36k^2}, & \|[\sqrt{[\varphi]},F]\|&<\frac{\mu}{36k^2}, & \|[h[\varphi],F]\|&<\frac{\mu}{3}.
\end{align*}
Let $a_1=\varphi(1)\left(a-\eps\right)_+$. Denote  $r = 1-\varphi(1)$ and set $a_2 = r^{1/2}\left(a-\eps\right)_+r^{1/2}$. We have
$$\|\left(a-\eps\right)_+ -(a_1+a_2)\|<\frac{\mu}{3}$$
since $r^{1/2}=1-h[\varphi](1)$.
We denote $g_{ij} = g([\varphi])(e_{ij})$.
We now define $\hat{c}_{ij} = \sqrt{[\varphi]}(1)g_{ij}c_{ij}$ and $\hat{c} = \sum_{i,j=1}^k \hat{c}_{ij}$.

Our goal is to first show that $\|\hat c \left(b-\delta\right)_+ \hat c ^* - a_1\|<\frac{\mu}{3}$:
\begin{align*}
\hat{c} \left(b-\delta\right)_+ \hat{c}^* &= \sum_{i,j,m,\ell=1}^k \hat c_{ij}\left(b-\delta\right)_+ \hat c_{m\ell}^*\\
& = \sqrt{[\varphi]}(1)\left(\sum_{i,j,m,\ell=1}^k g_{ij}c_{ij}\left(b-\delta\right)_+c_{m\ell}^*g_{\ell m}\right)\sqrt{[\varphi]}(1) \\
& \approx_{\mu/6} \sum_{i,j,m,\ell=1}^k g_{ij}g_{\ell m}c_{ij}\left(b-\delta\right)_+c_{m\ell}^*\\
& = \varphi(1)g[\varphi](1)\sum_{i,j,m=1}^k g_{im}c_{ij}\left(b-\delta\right)_+c_{mj}^* \\
& = \varphi(1)g[\varphi](1)\sum_{i,m=1}^k g_{im}\sum_{j=1}^kc_{ij}\left(b-\delta\right)_+c_{mj}^*\\
& = \varphi(1)g[\varphi](1)\sum_{i=1}^k g_{ii}(a-\eps)_+\\
& = \varphi(1)g^2[\varphi](1)(a-\eps)_+\\
& \approx_{\mu/6} \varphi(1)(a-\eps)_+\\
& = a_1
\end{align*} 
where the first approximation is due to our choice of $\eta$ and the second follows from (\ref{eq:1.2}).

We now deal with $a_2$. Since $a_2 \precsim r \precsim d$, we may find $s \in \A$ such that $\lnorm sds^*- a_2 \rnorm<\mu/3$. Furthermore, by replacing $s$ with $sp(b)$ where $p\in C_0([0,1])$ is some function that is $1$ on $[0, \delta/2]$ and vanishes on $[\delta, 1]$, we may assume that $s\left(b-\delta\right)_+ = 0$. We take $z = \hat c + s$ and calculate:
\begin{multline*}
\lnorm z\left[\left(b-\delta\right)_++d\right]z^*-\left(a-\eps\right)_+\rnorm=\lnorm\hat c\left(b-\delta\right)_+\hat c^*+sds^*- \left(a-\eps\right)_+ \rnorm\\
	\leq \lnorm\hat c \left(b-\delta\right)_+\hat c^*-a_1\rnorm+\lnorm sds^*-a_2\rnorm + \lnorm a_1 + a_2-\left(a-\eps\right)_+\rnorm<\mu.
\end{multline*}
Since this holds for every $\mu>0$ we have that $\left(a-\eps\right)_+ \precsim b$. Finally, since $\eps$ is arbitrary, we have $a \precsim b.$
\end{proof}

\begin{Thm}
\label{thm:almostunperf}
Let $\A$ be a simple unital $C^*$-algebra. If $\A$ is \TZA~ then $W(\A)$ is almost unperforated, and therefore $\A$ has strict comparison.
\end{Thm}
\begin{proof}

Let $a,b \in M_n(\A)_+$ such that $k\lb a \rb \leq (k-1)\lb b \rb$ for some $k$. We want to show that $\lb a \rb \leq \lb b \rb$. Since $M_n(\A)$ is also \TZA, we may assume without loss of generality that $a,b \in \A$.
If $b$ is purely positive then we are through because in particular $k\lb a \rb \leq k \lb b \rb$ and now we can apply Lemma $\ref{lemma:almostunperf}$.

If $b$ is not purely positive then $b$ is Cuntz equivalent to a projection $p \in \A$ and then by Lemma $\ref{lemma:notprojection}$ we may find a purely positive element $b^\prime \in \A$ such that $(k-1)\lb p \rb \leq k\lb b^\prime \rb$ and $b^\prime \leq p$. We may now apply Lemma $\ref{lemma:almostunperf}$ to get $a \precsim b^\prime \leq p \sim b$ 
\end{proof}

\section{Tracial $\Zh$ absorption and $\Zh$-absorption in the nuclear case}
 \label{section:nuclear}

The main result of this section is the following theorem.
\begin{Thm} \label{thm:tzaimpliesza}
Let $\A$ be a simple, separable, unital, nuclear \Csalg. If $\A$ is \TZA~ then $\A \cong \A \otimes \Zh$. 
\end{Thm}

The following theorem is a minor modification of results from \cite{matui-sato}.

\begin{Thm} \label{thm:ms}
Let $\A$ be a unital, separable, simple, nuclear \Csalg~ such that $\A$ has strict comparison, $T(\A)\neq\emptyset$, and the following condition holds:

For any $k\in \N$ there exists a c.p.c. order zero map $\psi:M_k \to \A_\infty\cap \A^\prime$ and a representative
$$(c_n)_{n\in\N}\in \prod_{n\in\N}\A$$
of $\psi(e_{1,1})\in\A_\infty$ such that $c_n\in \A$ is a positive contraction for all $n$ and
$$\lim_{n\to\infty} \max_{\tau\in T(\A)} |\tau(c_n^m)-1/k|=0.$$
Then $\A \cong \A\otimes\Zh$.
\end{Thm}
\begin{proof}
The proof of $(ii) \Rightarrow (iii)$ of Theorem 1.1 of \cite{matui-sato} now shows any c.p. map $\varphi:\A\to \A$ can be excised in small central sequences (See Definition 2.1 of \cite{matui-sato}). For this step we have used the hypothesis of the theorem and the fact that $T(\A)\neq \emptyset$ in order to apply Proposition 2.2 of \cite{matui-sato}. This implies that $\A$ has property (SI). Now the proof of $(iv) \Rightarrow (i)$ of Theorem 1.1 of \cite{matui-sato} shows that $\A$ is $\Zh$ absorbing. For this we have once again used the hypothesis of the theorem instead of Lemma 3.3 of \cite{matui-sato}.
\end{proof}

\begin{Lemma} \label{lemma:sc}
Let $\A$ be a simple, separable, unital, infinite dimensional \Csalg. For any $n$ there exists a positive contraction $c\in \A$ such that $c\neq 0$ and $d_\tau(c)\leq 1/n$ for all $\tau \in T(\A)$.
\end{Lemma}

\begin{proof}
Let $\theta:CM_n \to \A$ be an embedding as given by Proposition 4.10 of \cite{kirchberg-rordam}. We set $c=\theta(z \otimes e_{1,1})$, and it is immediate that $c$ satisfies the required property.
\end{proof}

\begin{Lemma} \label{lemma:tracialalmostdivisibility}
Let $\A$ be a simple, separable, stably finite, unital \Csalg. If $\A$ is \TZA~ then for any $k\in \N$, we can find a sequence of c.p.c. order zero maps $\varphi_n:M_k\to \A$ such that \begin{enumerate}
\item \label{property:1} $\displaystyle \lim_{n\to \infty} \max_{\tau \in T(\A)} |\tau(\varphi_n(e_{1,1})^m)-1/k|=0$ for all $m\in \N$.
\item \label{property:2} $\displaystyle \lim_{n\to\infty}\lnorm[a,\varphi_n(x)]\rnorm = 0$ for all $a\in\A$ and for all $x\in M_k$.
\end{enumerate}
\end{Lemma}
\begin{proof}
Using Lemma \ref{lemma:sc} we can find a sequence of non-zero positive contractions $c_n\in \A$ such that
$$
\lim_{n\to\infty} \max_{\tau\in T(\A)} d_\tau(c_n) = 0.
$$ 
Let $(a_n)_{\in\N}\subseteq \A$ be a dense sequence. We will denote $F_n=\{a_1,\ldots,a_n\}$. For each $n$ we can now find a c.p.c. order zero map $\varphi_n:M_k \to \A$ such that
$$
\|[F_n,\varphi_n]\|\leq 1/n
$$ and 
$$
1 - \varphi_n(1) \precsim c_n.
$$
Clearly property (\ref{property:2}) holds, so it remains to see that the maps $\varphi_n$ satisfy (\ref{property:1}). It is clear that that $\tau(\varphi_n(e_{1,1})^m) = \tau(\varphi_n(e_{i,i})^m)$ for any trace $\tau$, so it suffices to show that
$$1 - \max_{\tau \in T(\A)} \tau\left(\sum_{i=1}^k \varphi_n(e_{i,i})^m\right) =\min_{\tau \in T(\A)} \tau\left(1 - \sum_{i=1}^k \varphi_n(e_{i,i})^m\right) \underset{n\to\infty}{\longrightarrow} 0.$$
We claim that 
$$1 - \sum_{i=1}^k \varphi_n(e_{i,i})^m \sim 1 - \varphi_n(1)$$
for all $n$. To see this, fix $n$, and note that by functional calculus we have a homomorphism $\pi:C([0,1])\to \A$ given by
$$\pi(f) = f(\varphi_n(1)).$$
Letting $z$ denote the identity function on $[0,1]$, a simple calculation yields:
\begin{align*}
\pi (1-z^m) &= 1 - \sum_{i=1}^k \varphi_n(e_{i,i})^m, & \pi(1-z) = 1-\varphi_n(1)
\end{align*}
Since homomorphisms of \Csalg s induce maps on the level of the Cuntz semigroup, our claim now follows from the simple fact that $1-z^m \sim 1-z$ in $C([0,1])$.
 
We thus have 
\begin{align*}
0\leq \tau\left(1 - \sum_{i=1}^k \varphi_n(e_{i,i})^m\right) &\leq d_\tau\left(1 - \sum_{i=1}^k \varphi_n(e_{i,i})^m\right)\\& = d_\tau(1-\varphi(1)) \\ & \leq d_\tau(c_n)   
\end{align*}
and since $\displaystyle \lim_{n\to\infty} \max_{\tau\in T(\A)} d_\tau(c_n) = 0$, we have that (\ref{property:1}) holds.
\end{proof}

\begin{proof}[Proof of Theorem \ref{thm:tzaimpliesza}]
By Theorem \ref{thm:almostunperf} we know that $\A$ has strict comparison. Thus, if $\A$ is traceless then $\A$ is purely infinite. By Kirchberg's $\mathcal{O}_\infty$ absorption Theorem $\A\cong\A\otimes\mathcal{O}_\infty$ so by \cite{winter2} we are done.

Let us then assume that $T(\A)\neq\emptyset$. This implies that $\A$ is stably finite so the conclusion of Lemma \ref{lemma:tracialalmostdivisibility} holds. We now apply Theorem \ref{thm:ms} to complete the proof.
\end{proof}

\section{Finite group actions}
\label{section:finite-group}

We begin this section with a general lemma that we will need in this section and the following. We will only use it for actions of finite groups or of the integers.

\begin{Lemma}
\label{lemma:crossedproductcuntzdomination}
Let $\alpha:G\to \aut(\A)$ be an action of a discrete group $G$ on a simple, unital \Csalg~ $\A$. Suppose that $\alpha_g$ is outer for all $g\in G\setminus\{1\}$. Then for every non-zero positive element $a\in \A\rtimes_{\alpha,r} G$ in the reduced crossed product there exists a non-zero positive element $b\in \A$ such that $b\precsim a$.
\end{Lemma}
\begin{proof}
Let $E:\A\rtimes_{\alpha, r} G \to \A$ be the canonical faithful conditional expectation. By replacing $a$ with $\|E(a)\|^{-1} a$ we may assume that $\|E(a)\|=1$. We can find a finite sum $a^\prime = \sum_{k=1}^n a_ku_{g_k} \in \A\rtimes_{\alpha, r} G$ such that $\|a-a^\prime\|<\frac{1}{3}$ where $u_g$ is the canonical unitary implementing $\alpha_g$ for $g\in G$. We can assume that $a^\prime$ is positive and non-zero, and we label the indices such that $g_1 = 1_G$. Note that $a_1$ is a non-zero positive element in $\A$ because $a_1=E(a^\prime)$. Since $E$ is a contractive map, we have $$\|a_1\| \geq 1-\frac{1}{3} = \frac{2}{3}.$$
Let $0<\eps<\frac{1}{3(n+1)}$. By Lemma 3.2 of \cite{kishimoto} we can find a positive element $x\in\A$ with $\|x\|=1$ such that
\begin{align*}
\|xa_1x\|&>\|a_1\|-\eps \geq \frac{2}{3} - \eps, & \|xa_{g_k}\alpha_{g_k}(x)\|&<\eps, \quad \forall k\neq 1.
\end{align*}
This implies that for $k\neq 1$ we have
$$\|xa_{g_k}u_{g_k}x\| = \|xa_{g_k}\alpha_{g_k}(x)u_{g_k}\| = \|xa_{g_k}\alpha_{g_k}(x)\| < \eps.$$
We deduce that $\|xa^\prime x-xa_1x\|< n\eps$ and thus
$$\|xax-xa_1x\|<n\eps+\frac{1}{3}.$$
Let $b=(xa_1x-(n\eps+\frac{1}{3}))_+$, note that $b$ is a non-zero positive element by choice of $\eps$ and we have
$$b\precsim xax \precsim a$$
where the first Cuntz subequivalence follows from \cite{rordamuhf2}, Proposition 2.2.
\end{proof}

\begin{Def}\label{def:generalizedtracialrokhlinproperty}
If $\alpha:G\to \aut(\A)$ is an action of a finite group $G$ on a simple, unital \Csalg~ $\A$, then $\alpha$ is said to have the \emph{generalized tracial Rokhlin property} if for any $\eps>0$, any finite subset $F\subseteq \A$, and any non-zero positive element $a\in \A$ there exist normalized positive contractions $\{e_g\}_{g\in G} \subseteq \A$ such that:
\begin{enumerate}
\item $e_g\perp e_g$ when $g\neq h$.
\item $1-\sum_{g\in G} e_g \precsim a$.
\item $\|[e_g,y]\|<\eps$ for all $g\in G, y\in F$
\item $\|\alpha_g(e_h)-e_{gh}\|<\eps$ for all $g,h\in G$
\end{enumerate}
\end{Def}

We first collect a few basic properties of such actions.
\begin{Prop}
Let $\alpha:G\to \aut(\A)$ be an action of a finite group $G$ on a simple, unital \Csalg~ $\A$. If $\alpha$ has the generalized tracial Rokhlin property then $\alpha_g$ is outer for all $g\in G\setminus\{1\}$.
\end{Prop}
\begin{proof}
Let $g\in G$ and assume that $\alpha_g = \ad(u)$ for some $u\in \A$.
Let $\{e_g\}_{g\in G}$ be elements as in Definition \ref{def:generalizedtracialrokhlinproperty} taking $\eps=1/2$, $a=1_\A$, and $F=\{u\}$.
We have
\begin{align*}
1 &= \|e_1 - e_g\| \\
&< \|ue_1u^*-e_g\| + \frac{1}{2} \\
&= \|\alpha_g(e_1) - e_g\| + \frac{1}{2} \\
&< \frac{1}{2} + \frac{1}{2} \\
&= 1
\end{align*}
which is a contradiction.
\end{proof}

A special case of Theorem 3.1 of \cite{kishimoto} states that if $\alpha:G\to \aut(\A)$ is an action of a discrete group on a simple \Csalg~ such that $\alpha_g$ is outer for all $g\neq 1$ then $\A\rtimes_{\alpha,r} G$ is simple.
\begin{Cor}
Let $\alpha:G\to \aut(\A)$ be an action of a finite group $G$ on a simple, unital \Csalg~ $\A$. If $\alpha$ has the generalized tracial Rokhlin property then $\A\rtimes_\alpha G$ is simple. 
\end{Cor}

\begin{Lemma} \label{lemma:approximatelyinvarianttza}
Let $\A$ be a simple, separable, unital \TZA~ \Csalg~ and let $\alpha:G\to \aut(\A)$ be an action of a finite group $G$ on $\A$. Assume that $\alpha$ has the generalized tracial Rokhlin property. Then for any finite set $F \subset \A$, $\eps>0$ and non-zero positive element $a \in \A$ and $n \in \N$ there is a c.p.c. order zero map $\psi:M_n \to \A$ such that:
\begin{enumerate}
\item $1-\psi(1) \precsim a$. 
\item For any normalized element $x \in M_n$ and any $y\in F$ we have $\|[\psi(x),y]\| < \eps$.
\item For any normalized element $x \in M_n$ and any $g\in G$ we have $\|\alpha_g(\psi(x))-\psi(x)\| < \eps$.
\end{enumerate}
\end{Lemma}

\begin{proof}
Let $F,a,\eps,n$ be given as in the statement of the lemma. Since $\A$ is simple and unital, we can find $c>0$ such that $d_\tau(a)\geq c$ for all $\tau\in T(\A)$. Use Lemma \ref{lemma:sc} to find a positive element $b\in \A$ such that $d_\tau(b)<\frac{c}{|G|+2}$ for all $\tau\in T(\A)$. 
Since $\tau \circ \alpha_g$ is also a trace, we have that
\begin{align}\label{eq:4.2} d_\tau(\alpha_g(b))<\frac{c}{|G|+2}\end{align}
for all $g\in G$ and $\tau\in T(\A)$.

Let $\eta>0$ be as in Lemma \ref{lemma:orderzerofunctionalcalculus} such that $\|[h[\psi],y]\|<\eps/2$ whenever $\psi:M_k\to\A$ is a c.p.c order zero map and $y\in \A$ is a contraction such that $\|[\psi,y]\|<\eta$ where $h\in C_0((0,1])$ is given by
$$h(x)=\begin{cases} 2t, & t<1/2 \\ 1, & t\geq 1/2 \end{cases}.$$
Let $(e_g)_{g\in G}$ be positive contractions such that
\begin{enumerate}
\item $\|[e_g,y]\|<\eta$ for all $g\in G$ and for all $y\in F$.
\item $\|\alpha_g(e_h) - e_{gh}\|< \frac{\eta}{|G|}$.
\item $1-\sum_{g\in G} e_g \precsim b$
\end{enumerate}
as guaranteed by the generalized tracial Rokhlin property.

By Lemma \ref{lemma:centralsequencealgebra} we can find a c.p.c. order zero map $\varphi:M_n\to \A_\infty\cap \A^\prime$ such that $1-\varphi(1)\precsim b$.
We denote by $\bar{\alpha}_g$ the automorphism of $\A_{\infty}$ induced by $\alpha_g$, and define $\hat\psi:M_n\to \A_\infty$ by
$$
\hat\psi(x) = \sum_{g\in G} e_g\bar{\alpha}_g(\varphi(x)).
$$
$\hat{\psi}$ is clearly an order zero c.p.c. map.
 
First we remark that $\A_{\infty}\cap\A^\prime$ is invariant under $\bar{\alpha}$. Thus we have 
$$\hat\psi(x)y = \sum_{g\in G} e_g\alpha_g(\varphi(x))y \approx_{\eta} \sum_{g\in G} ye_g\alpha_g(\varphi(x)) = y\hat\psi(x)$$
for all $y\in F$. This is the first property we will be needing.

We now consider $\bar{\alpha}_g(\hat\psi(x))$ for normalized $x\in M_n$:
\begin{align*}
\bar{\alpha}_g(\hat\psi(x)) & = \bar{\alpha}_g\left(\sum_{h\in G}e_h\bar{\alpha}_h(\varphi(x))\right) \\
& \approx_{\eps/2} \sum_{h\in G}e_{gh} \bar{\alpha}_{gh}(\varphi(x)) \\
& = \sum_{g\in G}e_{g} \bar{\alpha}_{g}(\varphi(x)) = \hat\psi(x)
\end{align*}

The third and last property we will be needing of $\hat\psi$ is:
\begin{align*}
1-\hat\psi(1) & = \left(1-\sum_{g\in G} e_g\right) + \left(\sum_{g\in G} e_g - \hat\psi(1)\right) \\
& = \left(1 - \sum_{g\in G} e_g\right) + \sum_{g\in G} e_g^{1/2} \bar{\alpha}_g(1-\varphi(1))e_g^{1/2} \\
& \precsim b \oplus \bigoplus_{g\in G} \alpha_g(b) \precsim a
\end{align*}
where the last Cuntz-subequivalence follows from (\ref{eq:4.2}) and Theorem \ref{thm:almostunperf}. By \cite{rordamuhf2} Proposition 2.4, we can find $v\in\A_\infty$ such that $v^*av = (1-\hat\psi(1)-1/4)_+$.

We now use liftability of order zero maps with finite dimensional domains to lift $\hat\psi$ to a sequence of order zero maps $\psi_\ell:M_k\to \A$. Let $(v_\ell)_{\ell\in\N}\in\prod_\N \A$ be a representative of $v$. We have $$\|v_\ell^*av_\ell-(1-\psi_\ell(1)-1/4)_+\|\stackrel{\ell\to\infty}{\longrightarrow} 0.$$
The properties of $\hat\psi$ that we have established imply that for $\ell_0\in \N$ large enough, the following properties hold:
\begin{enumerate}
\item \label{property4.1} $\|[\psi_{\ell_0},y]\|<\eta$ for all $y\in K$,
\item \label{property4.2} $\|\alpha_g(\psi_{\ell_0}(x)) - \psi_{\ell_0}(x)\|<\eta$ for all normalized elements $x\in M_k$,
\item \label{property4.3} $\|v_{\ell_0}^*av_{\ell_0}-(1-\psi_{\ell_0}(1)-1/4)_+\|<1/4$,
\end{enumerate}
By Proposition 2.2 of \cite{rordamuhf2}, the third property above implies that
$$(1-\psi_{\ell_0}(1)-1/2)_+ = ((1-\psi_{\ell_0}(1)-1/4)_+-1/4)_+ \precsim v_{\ell_0}^*av_{\ell_0} \precsim a.$$
Define $\psi = h[\psi_{\ell_0}]$. By functional calculus, one checks that
$$1-\psi(1) = 2(1-\psi_{\ell_0}(1)-1/2)_+$$
which gives
$$1-\psi(1)\precsim a.$$
Property (\ref{property4.2}) above together with \ref{lemma:orderzerofunctionalcalculus2} ensures that $\|\alpha_g(\psi(x)) - \psi(x)\|<\eps/2<\eps$ for all normalized $x\in M_k$ and $g\in G$. This shows that $\psi$ meets our requirements.
\end{proof}

We are now ready to prove the main theorem of this section.

\begin{Thm}\label{thm:finitegroupaction}
Let $\A$ be a simple unital \Csalg~ and let $\alpha:G\to \aut(\A)$ be an action of a finite group on $\A$. Assume that $\alpha$ has the generalized tracial Rokhlin property. If $\A$ is \TZA~ then $\A\rtimes_\alpha G$ is also \TZA.
\end{Thm}
\begin{proof}
Lemma \ref{lemma:approximatelyinvarianttza} implies that we can always find c.p.c. order zero maps as in Definition \ref{def:TZA} provided that the positive element $a$ is taken from $\A$. By Lemma \ref{lemma:crossedproductcuntzdomination} we may always assume without loss of generality that $a \in \A$ (otherwise replace it by a non-zero positive element from $\A$ which it dominates).
\end{proof}

Putting together Theorems \ref{thm:tzaimpliesza} and \ref{thm:finitegroupaction} we obtain the following partial generalization of 
 Corollary 2.4 of \cite{hirshberg-winter}. 

\begin{Cor}
\label{cor:permanence-Z-finite-group}
Let $\A$ be a simple, separable, unital, nuclear \Csalg~ and let $\alpha:G\to \aut(\A)$ be an action of a finite group $G$ with the generalized tracial Rokhlin property. If $\A$ is $\Zh$ absorbing then so is $\A\rtimes_\alpha G$.
\end{Cor}

As a non-trivial example of an action satisfying the generalized tracial Rokhlin property, we consider the symmetric group acting by permutation on the $n$-fold tensor power of $\Zh$. This example was studied in \cite{hw-permutations}, Corollary \ref{cor:permanence-Z-finite-group} gives a different proof of those results.

We recall the following special case of \cite{rordam} Theorem 2.1 (which originates in the paper of Jiang and Su in \cite{jiang-su}).

\begin{Thm}\label{thm:selectivetraceembedding}
Let $\tau$ denote the unique tracial state on $\Zh$. There exists a unital embedding $\psi:C([0,1]) \to \Zh$ such that 
$$\tau(\psi(f)) = \int_0^1 f(t)dt$$
\end{Thm}

\begin{Cor}\label{approximatelycentralselectivetraceembedding}
Let $\tau$ denote the unique tracial state on $\Zh$, let $F\subseteq \Zh$ be a finite subset, and let $\eps>0$. There exists a unital embedding $\psi:C([0,1]) \to \Zh$ such that 
$$\tau(\psi(f)) = \int_0^1 f(t)dt$$ and $\|[y,\psi(f)]\|<\eps$ for all $y\in F$ and normalized $f\in C([0,1])$.
\end{Cor}
\begin{proof}
We may assume that all elements of $F$ are of norm at most 1. Now, we may decompose $\Zh$ as $\Zh \cong \Zh \otimes \Zh$ such that each element $y \in F$ is close to within $\eps/2$ to an element $y'$ in the unit ball of $\Zh\otimes 1$.  Now choose a unital embedding $\psi_0:C([0,1]) \to \Zh$ as in Theorem \ref{thm:selectivetraceembedding}. Define $\psi:C([0,1]) \to \Zh \otimes \Zh \cong \Zh$ by $\psi(f) = 1 \otimes \psi_0(f)$. $\psi$ clearly satisfies the commutation requirement, and since the restriction of the unique trace $\tau$ on $\Zh \otimes \Zh$ to the second component $1 \otimes \Zh$ is the unique trace on that $C^*$-algebra, we see that $\tau(\psi(f))$ is indeed given by integrating $f$ against Lebesgue measure, as required.
\end{proof}

We thank Dawn Archey for a helpful conversation relating to the following example.
\begin{Exl}
Let $G=S_n$ be the symmetric group on the set $\{1,\ldots,n\}$, and let $\A = \Zh^{\otimes n} \cong \Zh$. Let $\alpha:G\to\aut(\A)$ be the action defined on elementary tensors by the formula 
$$\alpha_\sigma(z_1\otimes\cdots\otimes z_n) = z_{\sigma^{-1}(1)}\otimes\cdots\otimes z_{\sigma^{-1}(n)}$$
Let $\eps, F, a$ be as in Definition \ref{def:generalizedtracialrokhlinproperty}. We denote $\delta = d_\tau(a)$ where $\tau$ is the unique trace on $\A$. 
Let $$f_{1}\in C([0,1]^n)$$
be a continuous function of norm one, such that: \begin{enumerate}
\item $f_{1}$ is supported on the set
$X:=\{(x_1,\ldots,x_n)\in [0,1]^n \mid x_1<\ldots<x_n\}$,
\item $\lambda(\{x\in X \mid\  f_{id}(x) \neq 1\}) < \frac{\delta}{n!}$
\end{enumerate}
where $\lambda$ is Lebesgue measure on $[0,1]^n$. For $\sigma \in G$ we now define $f_\sigma \in C([0,1]^n)$ by
$$f_{\sigma}(x_1,\ldots,x_n) = f_{1}(x_{\sigma(1)},\ldots,x_{\sigma(n)})$$
The $f_{\sigma}$ are pairwise orthogonal and sum up to an element of norm $1$ that equals $1$ on a subset of $[0,1]^n$ of measure larger than $1-\delta$.

Our next step will be to define an embedding of $C([0,1]^n)$ into $\A$. We may assume that the finite set $F$ consists of elementary tensors. Denote $F=\{z_1^{(i)}\otimes\cdots\otimes z_n^{(i)}\}_{i\in I}$ where $I$ is some finite index set. For $k=1,\ldots,n$ we take unital embeddings $\psi_k:C([0,1]) \to \Zh$ such that 
$$\tau(\psi_k(g)) = \int_0^1 g(t)dt$$ and $\|[z_k^{(i)},\psi_k(g)]\|<\eps^\frac{1}{n}$ for all $i\in I$ and normalized $g\in C([0,1])$. Using the identification $C([0,1]^n)\cong C([0,1])^{\otimes n}$ we define $\psi:C([0,1]^n) \to \A$ by 
$$\psi(g_1\otimes\cdots\otimes g_n)=\psi_1(g_1)\otimes\cdots\otimes\psi_n(g_n)$$
We define $e_\sigma = \psi(f_\sigma)$. Clearly we have that $\|[e_\sigma,y]\|<\eps$ for all $y\in F$.
It is also not hard to see that $$\tau(\psi(f)) = \int_{[0,1]^n} f d\lambda$$ for any $f\in C([0,1]^n)$ which implies that for the elements $e_\sigma = \psi(f_\sigma)$ we have $d_\tau(1-\sum_{\sigma\in G} e_\sigma) < \delta$. Since $\A$ has strict comparison and $\tau$ is the only trace on $\A$ this entails $1-\sum_{\sigma\in G} e_\sigma \precsim a$.

\end{Exl}

\section{Actions of $\Z$}
 \label{section:single-auto}

\begin{Def}\label{def:tracialrokhlinforautomorphisms}
Let $\alpha$ be an automorphism of a simple, unital \Csalg~ $\A$. We say that $\alpha$ has the \emph{generalized tracial Rokhlin property} if for any finite set $F\subseteq\A$, any $\eps>0$, and $k\in\N$, and any non-zero positive element $a\in\A$ there exist normalized orthogonal positive contractions $e_1,\ldots,e_k\in \A$ such that the following holds:
\begin{enumerate}
\item $1-\sum_{i=1}^k e_i \precsim a$.
\item $\|[e_i,y]\|<\eps$ for all $i$ and for all $y\in F$
\item $\|\alpha(e_i)-e_{i+1}\|<\eps$ for all $1\leq i\leq n-1$.
\end{enumerate} 
\end{Def}
Note that in the definition we do not require that $\alpha(e_k)$ be close to $e_1$.
\begin{Notation}
Elements $e_1,\ldots,e_k$ as in Definition \ref{def:tracialrokhlinforautomorphisms} are said to satisfy the relations $\Rr(k,\eps,a,F)$.
\end{Notation}

We begin as we did in the previous section with some basic properties.

\begin{Prop}
Let $\A$ be a simple, unital \Csalg~ and let $\alpha\in \aut(\A)$ have the generalized tracial Rokhlin property. Then $\alpha^m$ is outer for all $m\in \Z\setminus\{0\}$. 
\end{Prop}
\begin{proof}
Suppose $\alpha^m$ is implemented by some unitary $u\in \A$. Let $e_1, \ldots, e_{m+1}\in\A$ be elements satisfying $\Rr(m+1, \frac{1}{2m+2}, 1_\A, \{u\})$. We have
\begin{align*}
1 &= \|e_1-e_{m+1}\| \\
&< \|ue_1u^* - e_{m+1}\| + \frac{1}{2m+2} \\
&= \|\alpha^m(e_1)- e_{m+1}\| + \frac{1}{2m+2} \\
&< \frac{1}{2} + \frac{1}{2m+2} \\
&< 1
\end{align*}
\end{proof}

\begin{Cor}
Let $\alpha \in \aut(\A)$ be an automorphism on a simple, unital \Csalg~ $\A$. If $\alpha$ has the generalized tracial Rokhlin property then $\A\rtimes_\alpha \Z$ is simple.
\end{Cor}
\begin{proof}
Immediate from Theorem 3.1 of \cite{kishimoto}.
\end{proof}

\begin{Lemma}\label{lemma:smallrokhlinelements}
Let $\A$ be a simple, separable, unital, \Csalg~ and let $\alpha \in \aut(\A)$ be an automorphism with the generalized tracial Rokhlin property. Recall that $\alpha$ acts naturally on $T(\A)$ via $\tau \mapsto \tau\circ\alpha$. Assume that there exists $m\in\N$ such that the action of $\alpha^m$ on $T(\A)$ is trivial. Then for any $c>0$ there exist $k_0\in\N$, such that whenever $k\geq k_0$, $0< \eps < \frac{1}{k}$, $b\in \A_+$ satisfies $k\langle b\rangle\leq \langle 1\rangle$,  $F\subseteq \A$ is some finite subset and $e_1,\ldots,e_k\in \A$ are elements satisfying the relations $\Rr(k,\eps,b,F)$ then $d_\tau(e_i)<c$ for all $\tau\in T(\A)$ and for all $i$.
\end{Lemma}
\begin{proof}
Choose $k_0$ such that $\frac{2m}{k_0}<c$. Let $k$, $\eps>0$, $b\in \A_+$ be as in the statement of the lemma and let $e_1,\ldots,e_k$ be elements satisfying $\Rr(k,\eps,b,F)$. We partition the set $I = \{1,\ldots,k\}$ into $m$ sets
$$I = \bigsqcup_{j=1}^m I_j$$
where $I_j = I\cap (j+m\Z)$. We have that
$$\|\alpha^m(e_i) - e_{i+m}\|<m\eps$$
which implies that 
$$(\alpha^m(e_i) - m\eps)_+ = \alpha^m((e_i-m\eps)_+) \precsim e_{i+m}$$
whenever $1\leq i \leq k-m$. Since $\alpha^m$ acts trivially on $T(\A)$, this entails that 
$$d_\tau((e_i-m\eps)_+) \leq d_\tau(e_{i+m})$$
for all $1\leq i \leq k-m$. Similarly we have that
$$d_\tau((e_i-m\eps)_+) \leq d_\tau(e_{i-m})$$
whenever $m+1 \leq i \leq k$. By iterating this argument we get that
$$d_\tau((e_i- k\eps)_+) \leq d_\tau(e_{i^\prime})$$
for all $1\leq j \leq m$ and for all $i,i^\prime\in I_j$. Note that $|I_j|\leq  \left\lceil \frac{k}{m} \right\rceil$ for all $j$. This implies that, if $i \in I_j$ we have
$$d_\tau((e_i-k\eps)_+) \leq \frac{1}{|I_j|}\sum_{i^\prime \in I_j} d_\tau(e_{i^\prime}) \leq \frac{m}{k}d_\tau(1) < \frac{c}{2}.$$
Denote $\eta=k\eps$ and let $g,f\in C_0((0,1])$ be defined by
\begin{align*}
g(t) &= \begin{cases} 0, & t\leq\eta \\ \frac{1}{1-\eta}(t-\eta), & t>\eta \end{cases}, &f(t)=t-g(t)=\begin{cases} t, & t\leq\eta \\ \frac{\eta}{1-\eta}(1-t), & t>\eta \end{cases}.
\end{align*}
Notice that $g(x)\sim (x-\eta)_+$ for all $x\in\A_+$. Together with the previous observations this implies that $$d_\tau(g(e_i))<\frac{c}{2}$$ for all $i$ and for all $\tau\in T(\A)$. Now define another function $h\in C_0((0,1])$ by
$$h(t)= \begin{cases} \frac{t}{1-t}, & t\leq\eta \\ \frac{\eta}{1-\eta}, & t>\eta \end{cases}.$$
We have that
$$h(t)^\frac{1}{2}(1-t)h(t)^\frac{1}{2} = h(t)(1-t) = f(t),$$
thus
$$h(e_i)^\frac{1}{2} (1-e_i) h(e_i)^\frac{1}{2} = f(e_i)$$
for all $i$. Since $e_j$ is orthogonal to $e_i$ for $i\neq j$, this implies that
$$h(e_i)^\frac{1}{2}\left(1 - \sum_{j=1}^k e_j\right)h(e_i)^\frac{1}{2} = f(e_i)$$
entailing that
$$f(e_i)\precsim 1 - \sum_{j=1}^k e_j \precsim b$$
for all $i$.
Finally we have that 
\begin{align*}
e_i &= f(e_i)+g(e_i) \\ 
& \precsim f(e_i) \oplus g(e_i) \\ 
& \precsim b \oplus g(e_i)
\end{align*}
In particular we have $$d_\tau(e_i) \leq d_\tau(b) + d_\tau(g(e_i)) < c$$
\end{proof}

\begin{Lemma}\label{lemma:approximatelyzinvarianttzaforaiautomorphisms}
Let $\A$ be a simple, separable, unital, \TZA~ \Csalg~ and let $\alpha \in \aut(\A)$ be an automorphism with the generalized tracial Rokhlin property. If $\alpha^m$ acts trivially on $T(\A)$ for some $m$, then for any finite set $F \subset \A$, $\eps>0$ and non-zero positive element $a \in \A$ and $n \in \N$ there is a c.p.c. order zero map $\psi:M_n \to \A$ such that:
\begin{enumerate}
\item $1-\psi(1) \precsim a$. 
\item For any normalized element $x \in M_n$ and any $y\in F$ we have $\|[\psi(x),y]\| < \eps$.
\item For any normalized element $x \in M_n$ we have $\|\alpha(\psi(x))-\psi(x)\| < \eps$.
\end{enumerate}
\end{Lemma}

\begin{proof}
Let $F,\eps, a, n$ be given as in the statement of the lemma. We assume throughout that $\eps<1$. As in the proof of Lemma \ref{lemma:approximatelyinvarianttza} we can find $c>0$ such that $d_\tau(a)>c$ for all $\tau \in T(\A)$. Let $M\in \N$ be such that $\sqrt\frac{4}{M} < \frac{\eps}{2}$. Denote $c^\prime = \frac{c}{2M+1}$. By Lemma \ref{lemma:smallrokhlinelements} we can find $k_0$ such that whenever $k\geq k_0$, $0 <\delta <\frac{1}{k}$, $b\in \A_+$ with $k\langle b \rangle \leq \langle 1\rangle$,  $K\subseteq \A$ is some finite subset, and $e_1,\ldots,e_k\in \A_+$ are elements satisfying $\Rr(k,\delta,b,K)$ then we have $d_\tau(e_i)<c^\prime$ for all $i$ and for all $\tau \in T(\A)$. We may furthermore assume $k_0>2M$.  Use Lemma \ref{lemma:sc} to find a positive element $b\in \A$ such that $d_\tau(b)<\min\{\frac{1}{k},\frac{c}{2M}\}$ for all $\tau\in T(\A)$. Choose $k\geq k_0$ and choose $\delta>0$ such that $\left(2k^2\delta + k\delta^2 + \frac{4}{M}\right)^{\frac{1}{2}} < \frac{\eps}{2}$. In particular $\delta<\frac{1}{k}$.Let $K=\{\alpha^{-i}(y)\mid y\in F, 1\leq i\leq k\}$

Let $e_1,\ldots,e_k \in \A$ be elements satisfying $\Rr(k,\delta,b,K)$. We now modify the elements $e_i$ to obtain new elements $f_i$ as follows:
$$f_i = \begin{cases}e_i, & M\leq i \leq k-M\\ \frac{i}{M} e_i, & 1\leq i< M \\ \frac{k-i}{M} e_i, & k-M < i\leq k  \end{cases}$$
By Lemma \ref{lemma:centralsequencealgebra} we can find a c.p.c. order zero map $\varphi:M_n\to \A_\infty \cap \A^\prime$ such that $1-\varphi(1)\precsim b$. We define $\hat\psi:M_n\to \A_\infty$ by
$$\hat\psi(x) = \sum_{i=1}^k f_i\bar\alpha^i(\varphi(x)).$$
$\hat\psi$ is clearly a c.p.c. order zero map. Our first goal is now to find a nice bound on $\|\hat\psi(x)-\bar\alpha(\hat\psi(x))\|$ for normalized $x\in M_n$. First, note that $\|\alpha(f_i)-f_{i+1}\|<\frac{2}{M}$ for all $i<k$ and also $\|f_1\|,\|f_k\|<\frac{1}{M}$. Now assume $i\neq j$, we have
\begin{align*}
\lnorm(\alpha(f_i)-f_{i+1})(\alpha(f_j)-f_{j+1})\rnorm & = \lnorm-\alpha(f_i)f_{j+1}-f_{i+1}\alpha(f_j)\rnorm \\
& \leq \|\alpha(f_i)f_{j+1}\| + \|f_{i+1}\alpha(f_j)\| \\
& \leq \|\alpha(e_i)e_{j+1}\| + \|e_{i+1}\alpha(e_j)\| \\
& \leq \|e_{i+1}e_{j+1}\| + \|e_{i+1}e_{j+1}\| + 2\delta \\
& = 2\delta.
\end{align*}
We use this estimate to obtain that for normalized $x\in M_n$
\begin{IEEEeqnarray*}{rCl}
\IEEEeqnarraymulticol{3}{l}{
\lnorm\bar\alpha\left(\sum_{i=1}^{k-1}f_{i}\bar\alpha^{i}(\varphi(x))\right)-\sum_{i=1}^{k-1}f_{i+1}\bar\alpha^{i+1}(\varphi(x))\rnorm^2
}\\ \quad\quad\quad\quad\quad\quad\quad\quad\quad\quad
& = & \lnorm\sum_{i=1}^{k-1}(\alpha(f_i)-f_{i+1})\bar\alpha^{i+1}(\varphi(x))\rnorm^2 \\
& \leq & \lnorm \sum_{\substack{1\leq i,j \leq k-1 \\ i\neq j}}(\alpha(f_i)-f_{i+1})(\alpha(f_j)-f_{j+1})\bar\alpha^{i+1}(\varphi(x))\bar\alpha^{j+1}(\varphi(x))\rnorm \\
&&+ \lnorm\sum_{i=1}^{k-1} (\alpha(f_i)-f_{i+1})^2(\bar\alpha^{i+1}(\varphi(x)))^2\rnorm \\
& \leq & 2k^2\delta + k\delta^2 + \frac{4}{M}.
\end{IEEEeqnarray*}
Now we are ready to consider the expression we are interested in. Let $x\in M_n$ be a normalized element. We have
\begin{align*}
\|\bar\alpha(\hat\psi(x)) - \hat\psi(x)\| &\leq \lnorm\bar\alpha\left(\sum_{i=1}^{k-1}f_{i}\bar\alpha^{i}(\varphi(x))\right)-\sum_{i=1}^{k-1}f_{i+1}\bar\alpha^{i+1}(\varphi(x))\rnorm \\
&\quad+ \lnorm f_1\bar\alpha(\varphi(1))\rnorm + \lnorm\bar\alpha(f_k\bar\alpha^k(\varphi(1)))\rnorm \\
& \leq \left(2k^2\delta + k\delta^2 + \frac{4}{M}\right)^{1/2} + \frac{2}{M} \\
& \leq \left(2k^2\delta + k\delta^2 + \frac{4}{M}\right)^{1/2} + \frac{2}{M} \\
& < \eps.
\end{align*}

Next, we claim that $1-\hat\psi(1)\precsim a$. To see this, first note that
\begin{align*}
\left(1-\sum_{i=1}^k f_i\right) &\precsim 1-\sum_{i=M}^{k-M} e_i\\
& \precsim b^{\oplus 2M-1}
\end{align*}
Using this, we may write
\begin{align*}
1-\hat\psi(1) &= 1-\varphi(1) +\varphi(1)\left(1-\sum_{i=1}^k f_i\right)\\
&\precsim 1-\varphi(1) \oplus \left(1-\sum_{i=1}^k f_i\right)\\
&\precsim b^{\oplus{2M}}
\end{align*}
Since $\A$ has strict comparison and $d_\tau(b)< \frac{c}{2M}$ for all $\tau \in T(\A)$, we have that $b^{\oplus{2M}}\precsim a$ which proves our claim.

The final condition we need on $\hat\psi$ is that $\|[\hat\psi(x),y]\|<\eps$ for all normalized $x\in M_n$ and for all $y\in K$. This follows immediately from our construction.

We now continue as in the proof of Lemma \ref{lemma:approximatelyinvarianttza} and use projectivity of order zero maps to lift $\hat\psi$ to a sequence of c.p.c. order zero maps from $M_n$ into $\A$. Going far enough along this sequence we obtain the map $\psi$ which we need.
\end{proof}

Combining this with Theorem \ref{thm:tzaimpliesza}, we obtain the following.

\begin{Thm}\label{thm:zaction}
Let $\A$ be a simple, separable, unital, \TZA~ \Csalg~ and let $\alpha \in \aut(\A)$ be an automorphism with the generalized tracial Rokhlin property. If $\alpha^m$ acts trivially on $T(\A)$ for some $m$ then $\A\rtimes_{\alpha} \Z$ is also \TZA.
\end{Thm}

\begin{Rmk}
The hypothesis that $\alpha^m$ acts trivially on $T(\A)$ for some $m$ is satisfied automatically in many cases, for instance if some power of $\alpha$ is approximately inner, or if $T(\A)$ has finitely many extreme points. We suspect that this condition can be relaxed. 
\end{Rmk}

\begin{Exl}
It is shown in \cite{sato} that the bilateral tensor shift automorphism $\alpha$ on $\Zh^{\otimes \infty}\cong \Zh$ satisfies what is there defined as the weak Rokhlin property. Since $\Zh$ has strict comparison, this implies that $\alpha$ has the generalized tracial Rokhlin property. Since $\Zh$ has unique trace this shows that $\alpha$ satisfies the conditions of Theorem \ref{thm:zaction}.
\end{Exl}

\end{document}